\documentclass[]{interact}
\usepackage[colorlinks = true,
            linkcolor = blue,
            citecolor = blue]{hyperref}
\theoremstyle{plain}
\usepackage{fullpage,yhmath}
\usepackage{hyperref}
\usepackage{amsmath,amsthm,amsopn,amstext,amscd,amsfonts,amssymb}
\usepackage{amsthm}
\usepackage{epstopdf}
\usepackage[caption=false]{subfig}
\usepackage{empheq}

\usepackage[numbers,sort&compress]{natbib}
\bibpunct[, ]{[}{]}{,}{n}{,}{,}
\makeatletter
\def\NAT@def@citea{\def\@citea{\NAT@separator}}
\makeatother

\theoremstyle{plain}
\newtheorem{theorem}{Theorem}[section]
\newtheorem{lemma}[theorem]{Lemma}
\newtheorem{corollary}[theorem]{Corollary}

\theoremstyle{definition}
\newtheorem{definition}[theorem]{Definition}

\theoremstyle{remark}
\newtheorem{remark}{Remark}

\def\dis{\displaystyle}
\def\ds{\displaystyle}

\DeclareMathOperator{\N}{\mathbb{N}}

\newcommand{\car}[1]{\raise1pt\hbox{$\chi$}_{#1}}

\usepackage{graphics}
\usepackage{pgfplots}
\pgfplotsset{width=7cm,compat=newest} 
\usepackage{graphics}
\usepackage{tikz}
\usetikzlibrary{shapes.geometric}
\usepackage{tikz-cd}
\usepackage[labelsep=endash]{caption}
\renewcommand{\thefigure}{\arabic{figure}}
\usepackage{float}
\begin{document}


\title{Anisotropic elliptic equations involving unbounded coefficients and singular nonlinearities}
\author{
\name{Fessel Achhoud\textsuperscript{a}\thanks{Fessel Achhoud Email: fessel.achhoud@studenti.unime.it} and Hichem. Khelifi\textsuperscript{b} \thanks{Hichem Khelifi Email: khelifi.hichemedp@gmail.com}}
\affil{\textsuperscript{a} Dipartimento di Matematica e Informatica, University of Catania, Viale A. Doria, 6, 95125 Catania, Italy \\
\textsuperscript{b} Department of Mathematics, University of Algiers 1, Benyoucef Benkhedda, 2 Rue Didouche Mourad, Algiers, Algeria.}
}
\maketitle
\begin{abstract}
In this paper, we study the existence and regularity of solutions for a class of nonlinear singular elliptic equations involving unbounded coefficients and a singular right-hand side. Specifically, we are interested  to problem whose simplest model  is
\begin{equation*}
-\sum_{j=1}^N\partial_{j}\left([1+u^{q}]\vert \partial_{j} u \vert^{p_{j}-2} \partial_{j} u\right)= \frac{f}{u^{\gamma}}\text{ in $\mathcal{D},$}\quad  u>0  \text{ in $\mathcal{D},$} \quad u=0  \hbox{ on}\;\; \partial\mathcal{D},
\end{equation*}
where $\mathcal{D}$ is a bounded open subset of $\mathbb{R}^{N}$ with $N>2$, $ \gamma\geq0$, $q >0 $, $p_{j}>2$ for all $j=1,...,N$ and the source term $f$ belongs to $L^1(\mathcal{D})$, with $f \geq 0$ and $f \not\equiv 0$.
\end{abstract}

\begin{keywords}
Anisotropic elliptic equations, Unbounded coefficients, Existence and regularity, Singular term,  $L^{1}$ data.
\end{keywords}

\section{Introduction}
 In the present work we investigate  the boundary value problem, given by
\begin{equation}
 \label{fhpr}
\left\{\begin{array}{ll} \displaystyle -\sum_{j\in\mathcal{E}}\partial_{j}\left([b(x)+ u^{q}]\vert \partial_{j} u \vert^{p_{j}-2} \partial_{j} u\right)= \frac{f}{u^{\gamma}} & \hbox{in}\;\;\mathcal{D}, \\
u>0 &\hbox{in}\;\;\mathcal{D},\\
u=0 & \hbox{on}\;\; \partial\mathcal{D},
 \end{array}
 \right.
\end{equation}
where $\mathcal{D}$ is a bounded open subset of $\mathbb{R}^{N}$ ($N\geq 3$), $\mathcal{E}=\{ j \in \mathbb{N} : 1 \leq j \leq N \}$, $q,\gamma>0$, and $p_{j}$ satisfies 
\begin{align}
&\label{(01.2)}
 p_{N}\geq p_{N-1}\geq \ldots \geq p_{2}\geq p_{1}\geq 2\quad \mbox{and}\quad N>\overline{p}\geq 2,
\end{align}
where $\overline{p}$ is the harmonic mean of $p_i$, defined as 
$$\overline{p}=\left(\frac{1}{N}\sum_{j\in\mathcal{E}}^{N}\frac{1}{p_{j}}\right)^{-1}.$$
\par $a: \mathcal{D} \rightarrow \mathbb{R}$ is a measurable function such that
\begin{equation}
\label{(1.2)}
\beta \geq b(x) \geq \alpha\;\; \mbox{a.e. in}\; \mathcal{D}, 
\end{equation}
with $\alpha, \beta>0$, and
\begin{equation}
\label{(1.3)}
f\in L^{1}(\mathcal{D}),  \quad \mbox{with } \quad f\geq 0, \quad \mbox{and} \quad f \not\equiv0. 
\end{equation}
There is by now a large number of papers and an increasing interest about anisotropic problems. With no hope of being complete, let us mention some pioneering works on anisotropic Sobolev spaces 
\cite{SNIM, SM, JR, JR1} and some more recent existence and regularity results for  anisotropic boundary value problems \cite{FR, KM, ZKM}.
Interest in anisotropic problems has significantly increased in recent years due to their wide range of applications in mathematical modelling of natural phenomena, particularly in biology and fluid mechanics. For instance, anisotropy plays a crucial role in the mathematical modelling of fluid dynamics in anisotropic media, where the conductivities of the material differ in different directions (see, e.g., \cite{SJS}). Additionally, anisotropic models are used in biology to describe the spread of epidemic diseases in heterogeneous environments (see, e.g., \cite{MK}). Numerous results have been achieved in the study of anisotropic problems across various fields. For further details, the reader is referred to \cite{AIAA,AGF,SM,BS}, as well as the references cited therein.

From the mathematical point of view, the anisotropic quasilinear elliptic equations of type \eqref{fhpr} naturally arise within variational frameworks, for instance when considering integral functionals like

\[
\mathcal{J}(v):=\sum_{j\in\mathcal{E}}\frac{1}{p_j}\int_{\mathcal{D}}[b(x)+|v|^{q}]|\partial_{j}v|^{p_j}dx -\int_{\mathcal{D}}F(x,v)\,dx,
\]

where \(F(x,v)\) suitably captures singular nonlinearities. Problems of this kind, whether or not directly linked to a variational formulation, have been extensively studied in recent literature under various hypotheses on data and coefficients (see, for example, \cite{ALF,L}). 

We have to mention that the investigation of an anisotropic elliptic problem involving a singular nonlinearity was initiated in \cite{LM}; specifically, the authors focus on proving the existence and regularity of solutions  to the following boundary value problem
\begin{equation}\label{miripr}
-\sum_{j=1}^N\partial_{j}\left(\vert \partial_{j} u \vert^{p_{j}-2} \partial_{j} u\right)= \frac{f}{u^{\gamma}}\text{ in $\mathcal{D},$}\quad  u>0  \text{ in $\mathcal{D},$} \quad u=0  \hbox{ on}\;\; \partial\mathcal{D},
\end{equation}

The singular nonlinearity arises from the term \( u^{-\gamma} \), where the exponent \( \gamma \) can vary. The paper investigates several cases of \( \gamma \), including \( \gamma = 1 \), \( \gamma < 1 \), and \( \gamma > 1 \).
 They showed that:

%
%

\begin{itemize}
\item[$(\textbf{A}_{\gamma})$] If $\gamma < 1$,  then $u\in W_{0}^{1, (s_{j})}(\mathcal{D})$ with $s_j<\ds\frac{N[\bar{p}-1+\gamma]}{[N-1+\gamma]\overline{p}}p_{j},\quad \forall\; j\in\mathcal{E}$.

\item[$(\textbf{B}_{\gamma})$]  If $\gamma > 1$, then $u\in W_{0}^{1, (p_{j})}(\mathcal{D})\cap L^{r(\gamma)}(\mathcal{D})$, where
$ r(\gamma)=\ds\frac{N(\gamma-1+\overline{p})}{N-\overline{p}}.$

\item[$(\textbf{C}_{\gamma})$] If $\gamma = 1$, then $u\in W_{0}^{1, (p_{j})}(\mathcal{D})$.
\end{itemize}

These results provide a foundation for understanding the regularity of solutions to anisotropic elliptic problems, particularly when singular non-linearities are present. In this paper, we extend and improve upon these findings by
%
%
deeply explore the interplay between an additional term of the form
$$\displaystyle -\sum_{j\in\mathcal{E}}\partial_{j}\left(u^{q}\vert \partial_{j} u \vert^{p_{j}-2} \partial_{j} u\right)$$	
 and the singular nonlinearity $\dis u^{-\gamma}$ in presence of data with really poor summability, namely $f\in L^1(\mathcal{D})$. In particular we deal with the regularizing effect,  in terms of Sobolev regularity,  provided by the lower order terms to the solutions of problems as  \eqref{miripr}. 
	
	These kinds of regularizing effects given by the gradient terms  with natural growth in isotropic elliptic problems  are nowadays quite classical see for instance \cite{FR,ALF,L}.

%
%
%

Namely, by \textit{distributional solution} for problem \eqref{fhpr}  a function $u \in W_0^{1,1}(\mathcal{D})$ that satisfies
\begin{align}
\label{M1}
\left(b(x)+ u^{q}\right)\vert  \partial_{j}u\vert^{p_{j}-1}\in L^{1}(\mathcal{D}),\quad \forall\; j\in\mathcal{E},
\end{align} 
and there exists $\tilde{C} > 0$ such that
\begin{align}
\label{M2}
 u\geq \tilde{C}\;\; \mbox{in}\; \mathcal{D},
\end{align}
Moreover, for all test functions $\varphi \in C_c^1(\mathcal{D})$, the weak formulation holds
\begin{align}
\label{M3}
\sum_{j\in\mathcal{E}}\int_{\mathcal{D}} [b(x)+u^{q}]\vert \partial_{j} u \vert^{p_{j}-2} \partial_{j} u \partial_{j}\varphi \;dx =\int_{\mathcal{D}}\frac{f}{u^{\gamma}}\varphi \;dx.
\end{align}

 Our main result is the following.
\begin{theorem} 
\label{theo1}
Let $f \in L^{1}(\mathcal{D})$ be nonnegative, and assume conditions \eqref{(01.2)}–\eqref{(1.2)} hold. Then, problem \eqref{fhpr} admits a nonnegative  weak solution $u$, with the following regularity
\begin{itemize}
\item[$(\textbf{A}_{\gamma,q})$] If $0 < \gamma < 1$:
\begin{itemize}
\item[(i)] When $q > 1 - \gamma$, the solution belongs to $W_{0}^{1, (p_{j})}(\mathcal{D})$.
\item[(ii)] When $q \leq 1 - \gamma$, the solution belongs to $W_{0}^{1, (\eta_{j})}(\mathcal{D})$, where
$$\eta_{j}=\frac{N(q+\gamma+\bar{p}-1)}{[N-(1-q-\gamma)]\overline{p}}p_{j},\quad \forall\; j\in\mathcal{E}.$$
\end{itemize}

\item[$(\textbf{B}_{\gamma,q})$]  If $\gamma > 1$ and $q \geq 0$, then $u\in W_{0}^{1, (p_{j})}(\mathcal{D})\cap L^{r(\gamma,q)}(\mathcal{D})$, where
$$ r(\gamma,q)=\frac{N(q+\gamma-1+\overline{p})}{N-\overline{p}}.$$
\item[$(\textbf{C}_{\gamma,q})$] If $\gamma = 1$ and $q \geq 0$, then $u\in W_{0}^{1, (p_{j})}(\mathcal{D})\cap L^{r(1,q)}(\mathcal{D})$.
\end{itemize}
\end{theorem}
\begin{remark}
We observe that $$q>0\Leftrightarrow r(\gamma,q)>r(\gamma),$$
and  $$\bar{p}<N \Leftrightarrow W^{1,\eta_j}_0(\mathcal{D})\subseteq W^{1,s_j}_0(\mathcal{D}).$$
\end{remark}
\begin{remark}
Note that
\begin{equation}
\gamma=0 \Rightarrow (\textbf{A}_{\gamma,q})\equiv \text{ Theorem 3.7 in \cite{LR} (with $\lambda=p$).}
\end{equation}
So, our result  links up continuously with the result in the isotropic case without the singular term in the right hand side.
\end{remark}
\begin{remark}

In this paper, we restrict our investigation of the problem \eqref{fhpr} to the case \( f \in L^1(\mathcal{D}) \). However, inspired by the results in the isotropic framework presented in \cite{FR}, it is possible to extend the regularization effect established here to the case with purely summable data, namely to functions \( f \in L^m(\mathcal{D}) \) with \(1 < m < (\bar{p}^*)' \).
\end{remark}

\textbf{Notation.} Hereafter we use the following standard notations: for $k > 0$ and all $s \in \mathbb{R}$, we define the truncation functions $T_k(s)$ and $G_k(s)$ as follows
\begin{equation*}
T_{k}(s)=\min\{k,\max\{-k,s\}\},\quad \mbox{and}\quad G_{k}(s)=s-T_{k}(s).
\end{equation*}

The paper is organized as follows. In Section 2  we provide some fundamental information for the theory of anisotropic Sobolev spaces since it is our work space. Section 3 is devoted to prove our main Theorem.
\section{Anisotropic Sobolev spaces}
The anisotropic Sobolev spaces $W^{1,(p_{j})}(\mathcal{D})$ and $W_{0}^{1,(p_{j})}(\mathcal{D})$ provide the functional framework for Problem \eqref{fhpr}; see \cite{R10, JR, JR1} for details. Let $\mathcal{D}$ be an open bounded subset of $\mathbb{R}^{N}$, where $N\geq 2$, and let $p_{j}$ satisfy the conditions \eqref{(01.2)}. The anisotropic Sobolev spaces $W^{1,(p_{j})}(\mathcal{D})$ and $W_{0}^{1,(p_{j})}(\mathcal{D})$ are defined as follows
$$\begin{array}{ll} \displaystyle W^{1,(p_{j})}(\mathcal{D})=\left\{z\in W^{1,1}(\mathcal{D}) :\; \int_{\mathcal{D}}\vert\partial_{j}z\vert^{p_{j}}\;dx<\infty,\;\;  j\in\mathcal{E}\right\},
\\
\displaystyle W_{0}^{1,(p_{j})}(\mathcal{D})=\left\{z\in W_{0}^{1,1}(\mathcal{D}) :\; \int_{\mathcal{D}}\vert\partial_{j}z\vert^{p_{j}}\;dx<\infty,\;\;  j\in\mathcal{E}\right\}.
\end{array}
$$
Alternatively, $W_{0}^{1,(p_{j})}(\mathcal{D})$ is the closure of $C_{0}^{\infty}(\mathcal{D})$ with respect to the norm
$$\Vert z\Vert_{W_{0}^{1,(p_{j})}(\mathcal{D})}=\sum_{j\in\mathcal{E}}
\left(\int_{\mathcal{D}}\vert\partial_{j}z\vert^{p_{j}}\;dx\right)^{\frac{1}{p_{j}}},$$
with this norm, $W_{0}^{1,(p_{j})}(\mathcal{D})$ is a separable and reflexive Banach space, and its dual is $\left( W_{0}^{1,(p_{j})}(\mathcal{D}) \right)^{*}$, where $ p'_{j}$ is the conjugate of $p_{j}$, i.e., $\frac{1}{p_{j}} + \frac{1}{p'_{j}} = 1$ for all $j\in\mathcal{E}$. 
\begin{lemma}\cite{R10}
\label{L1}
For any $z \in W_{0}^{1,(p_{j})}(\mathcal{D})$ with $\overline{p} < N$, there exists a constant $C > 0$ depending only on $\mathcal{D}$ such that
\begin{align}
\label{(4)}
&\left(\int_{\mathcal{D}}\vert z\vert^{\delta}\;dx\right)^{\frac{1}{\delta}}\leq C\prod_{j\in\mathcal{E}}\left(\int_{\mathcal{D}}\vert \partial_{j}z\vert^{p_{j}}\;dx\right)^{\frac{1}{Np_{j}}}, \quad \forall \delta\in[1,\overline{p}^{*}],\quad \overline{p}^{*}=\frac{N\overline{p}}{N-\overline{p}}, \\
\label{(5)}
&\left(\int_{\mathcal{D}}\vert z\vert^{\overline{p}^{*}}\;dx\right)^{\frac{p_{N}}{\overline{p}^{*}}}\leq C \sum_{j\in\mathcal{E}}\int_{\mathcal{D}}\vert \partial_{j}z\vert^{p_{j}}\;dx.
\end{align}
\end{lemma}
\begin{lemma}\cite{R10}
\label{L2}
For any $z\in W_{0}^{1,(\delta_{j})}(\mathcal{D})\cap L^{\infty}(\mathcal{D})$ with $\overline{\delta}<N$, there exists a constant $C>0$ depending only on $\mathcal{D}$ such that 
\begin{equation}
\label{(9)}
\left(\int_{\mathcal{D}}\vert z\vert^{r}\;dx\right)^{\frac{N}{\overline{\delta}}-1}\leq C\prod_{j\in\mathcal{E}}\left(\int_{\mathcal{D}}
\vert \partial_{j}z\vert^{\delta_{j}}\vert v\vert^{t_{j}\delta_{j}}\;dx\right)^{\frac{1}{\delta_{j}}},
\end{equation}
for any $r$ and $t_{j}\geq 0$ satisfying
$$\frac{1}{r}=\frac{b_{j}(N-1)-1+\frac{1}{\delta_{j}}}{t_{j}+1}\quad \mbox{and}\quad \sum_{j\in\mathcal{E}}b_{j}=1.$$
\end{lemma}
\par To prove that the solution $z$ in $\mathcal{D}$ is positive, we need the theorem, which we will discuss later. As a preliminary to this theorem, we take into account the following  problem
\begin{align}
\label{(10)}
\left\{\begin{array}{ll}-\displaystyle\sum_{j\in\mathcal{E}} \partial_{j}\left[\vert \partial_{j}z\vert^{p_{j}-2}\partial_{j}z\right]=\lambda \vert z\vert^{q-2}z& \hbox{in}\;\;\mathcal{D}, \\
 z  =0 & \hbox{on}\;\; \partial\mathcal{D}, \end{array}
 \right.
\end{align}
here $\lambda>0$ and $p_{1}<q<p_{N}$.
\par We define weak supersolutions for the given problem \eqref{(10)}; for further details, see \cite{R3}.
\begin{definition}  
\label{D1}  
A function $z \in W_{0}^{1,(p_{j})}(\mathcal{D})$ that is nonnegative a.e. in $\mathcal{D}$ is a positive weak supersolution of \eqref{(10)} if  
\begin{equation}  
\label{(11)}  
\sum_{j\in\mathcal{E}} \int_{\mathcal{D}} \vert \partial_{j} z \vert^{p_{j}-2} \partial_{j} z \, \partial_{j} \phi \;dx\geq 0, \quad \forall \phi \in C_{0}^{\infty}(\mathcal{D}), \; \phi \geq 0.  
\end{equation}  
\end{definition}  
\par The following weak Harnack inequality for weak super-solutions is the key result.
\begin{theorem}{\cite{R3,R4}}
\label{T1}
Suppose that $z$ is a weak non-negative supersolution of \eqref{(10)} with $z < M$ in $\mathcal{D}$, and that $p_1 \geq 2$. Then
\begin{equation}
\label{(12)}
\rho^{-\frac{N}{\beta}}\Vert z\Vert_{L^{\beta}(K(2\rho))}\leq C \min\limits_{K(\rho)}z,
\end{equation}
 for $\beta<\frac{N(p_{1}-1)}{N-p_{1}}$, if $p_{1}\leq N$, for any $\beta$, if $p_{1}>N$.
\end{theorem}
\par The following is obvious from the previous Theorem.
\begin{corollary}{\cite{R3}}
\label{C1}
Under the assumption that $p_{1} \geq 2$, any weak non-negative solution of \eqref{(10)} must either be strictly positive in $\mathcal{D}$ or identically zero.
\end{corollary}
 
\section{Proof of the Main result}
\subsection{Approximation of problem \eqref{fhpr}}
For $n \in \mathbb{N}$, we approximate problem \eqref{fhpr} by considering the following regularized (non-singular) problem
\begin{equation}
 \label{(2.1)}
\left\{\begin{array}{ll} \displaystyle -\sum_{j\in\mathcal{E}}\partial_{j}\left([b(x)+u_{n}^{q}]\vert \partial_{j} u_{n} \vert^{p_{j}-2} \partial_{j} u_{n}\right)= \frac{f_{n}}{\left(u_n +\frac{1}{n}\right)^{\gamma}} & \hbox{in}\;\;\mathcal{D}, \\
u_{n}>0 & \hbox{in}\;\; \mathcal{D},\\
u_{n}=0 & \hbox{on}\;\; \partial\mathcal{D},
 \end{array}
 \right.
\end{equation}
where $f_{n} = T_{n}(f)\in L^{\infty}(\mathcal{D})$ is a sequence of bounded functions that converges to $f > 0$ in $L^{1}(\mathcal{D})$. 
\begin{lemma}
\label{L2.1}
Under assumptions \eqref{(1.2)} and \eqref{(1.3)}, problem \eqref{(2.1)} admits a non-negative solution $u_n \in W_{0}^{1, (p_{j})}(\mathcal{D}) \cap L^{\infty}(\mathcal{D})$ satisfying
\begin{align}
\label{(2.2)}
&\displaystyle\sum_{j\in\mathcal{E}}\int_{\mathcal{D}} \left[b(x)+u_{n}^{q}\right]\vert \partial_{j} u_{n}\vert^{p_{j}-2} \partial_{j} u_{n}\partial_{j}\varphi \;dx =\int_{\mathcal{D}}\frac{f_{n}\varphi \;dx}{\left(u_{n}+\frac{1}{n}\right)^{\gamma}}, 
\end{align}
for all $\varphi\in W_{0}^{1,(p_{j})}(\mathcal{D})\cap L^{\infty}(\mathcal{D})$.
\end{lemma}
\begin{proof}
For fixed $n \geq 1$ and $k \geq 0$, let $v \in L^{\overline{p}}(\mathcal{D})$ and define $w = \mathcal{P}(v)$ as the unique solution of
\begin{equation} 
\label{(2.3)}
\left\{\begin{array}{ll}\displaystyle-\sum_{j\in\mathcal{E}}\partial_{j}\left(\left[b(x)+ T_{k}(v)^{q}\right]\vert \partial_{j} w\vert^{p_{j}-2} \partial_{j} w\right)=\frac{f_{n}}{\left(\vert v\vert+\frac{1}{n}\right)^{\gamma}}&\hbox{in}\;\mathcal{D}, \\
w=0 & \hbox{on}\; \partial\mathcal{D}, 
\end{array}
\right.
\end{equation}
Since problem \eqref{(2.3)} has a unique solution \cite{R1}, the operator $\mathcal{P}$ is well-defined. We aim to establish the existence of a fixed point for $\mathcal{P}$ via Schauder's fixed point theorem \cite{R2}.
\par  Picking $\varphi = w$ as a test function in \eqref{(2.3)}, and making use of \eqref{(1.2)} along with the fact that
$$
\left\vert\frac{f_{n}}{\left(\vert v\vert+\frac{1}{n}\right)^{\gamma}}\right\vert\leq n^{\gamma+1},\quad \vert T_{k}(v)\vert^{q}\geq 0,
$$
we have
$$
\alpha\sum_{j\in\mathcal{E}}\Vert \partial_{j}w\Vert_{L^{p_{j}}(\mathcal{D})}^{p_{j}}\leq n^{\gamma+1} \int_{\mathcal{D}} \vert w\vert \;dx.
$$
Using \eqref{(5)} on the left and H\"older’s inequality with exponent $\overline{p}^{\star}$ on the right, we obtain
\begin{align}
\label{(13)}
\Vert w\Vert_{L^{\overline{p}^{\star}}(\mathcal{D})}^{p_{N}}&\leq \sum_{j\in\mathcal{E}}\Vert \partial_{j}u_{n}\Vert_{L^{p_{j}}(\mathcal{D})}^{p_{j}}\leq C(n,\gamma,\alpha)\vert \mathcal{D}\vert^{\frac{1}{(\overline{p}^{\star})'}}\Vert w\Vert_{L^{\overline{p}^{\star}}(\mathcal{D})}.
\end{align}
Since $p_N > 1$, there exists a positive constant $R(n, |\mathcal{D}|, \gamma, \alpha)$ that is independent of $v$ and $w$, such that
\begin{align}
\label{(1414)}
\Vert w\Vert_{L^{\overline{p}^{\star}}(\mathcal{D})}\leq R(n,\vert\mathcal{D}\vert,\gamma,\alpha).
\end{align}
Observing that $\overline{p}<\overline{p}^{\star}$, then
\begin{align}
\label{(14)}
\Vert w\Vert_{L^{\overline{p}}(\mathcal{D})}\leq R(n,\vert\mathcal{D}\vert,\gamma,\alpha).
\end{align}
Therefore, equation \eqref{(14)} implies that the ball $B$ in $L^{\overline{p}}(\mathcal{D})$, with radius $R(n, |\mathcal{D}|, \gamma, \alpha)$, remains invariant under the mapping $\mathcal{P}$.\\
\par \textit{Observation 1:} $\mathcal{P}(L^{p}(\mathcal{D}))$ has relative compactness in $L^{\overline{p}}(\mathcal{D})$.\\
Using equations \eqref{(13)} and \eqref{(1414)}, we can deduce that
\begin{align*}
\sum_{j\in\mathcal{E}}\int_{\mathcal{D}}\vert \partial_{j}w\vert^{p_{j}}\;dx=\sum_{j\in\mathcal{E}}\int_{\mathcal{D}}\vert \partial_{j}\mathcal{P}(v)\vert^{p_{j}}\;dx\leq R(n,\vert\mathcal{D}\vert,\gamma,\alpha),\quad \forall v\in L^{\overline{p}}(\mathcal{D}).
\end{align*}
By Sobolev embedding, $\mathcal{P}(L^{\overline{p}}(\mathcal{D}))$ can be shown to be compact in $L^{\overline{p}}(\mathcal{D})$.
\par \textit{Observation 2:} $\mathcal{P}$ is a continuous operator.\\
Let $v_r,$ for any $r\in \N,$ be a sequence converging to $v$ in $L^{\overline{p}}(\mathcal{D})$. By the dominated convergence theorem, it follows that
$$\frac{f_{n}}{\left(\vert v_{r}\vert+\frac{1}{n}\right)^{\gamma}} \rightarrow \frac{f_{n}}{\left(\vert v\vert+\frac{1}{n}\right)^{\gamma}}\;\; \mbox{strongly in}\; L^{p}(\mathcal{D}).$$ 
By uniqueness, $w_{r}:=\mathcal{P}(v_{r})$ converges to $w:=\mathcal{P}(v)$ in $L^{\overline{p}}(\mathcal{D})$. Applying Schauder's fixed point theorem, for each fixed $n$, there exists $u_{n,k}\in W_{0}^{1,p}(\mathcal{D})$ such that $\mathcal{P}(u_{n,k})=u_{n,k}$, with $u_{n,k}\in L^{\infty}(\mathcal{D})$ for all $n,k\in \mathbb{N}$. Indeed, for fixed $h\geq 1$, using $G_{h}(u_{n,k})$ as a test function and noting that $u_{n,k}+\frac{1}{n}\geq h\geq 1$ on $\{u_{n,k}\geq h\}$, we obtain
$$\sum_{j\in\mathcal{E}}\int_{\mathcal{D}}\vert \partial_{j}G_{h}(u_{n,k})\vert^{p_{j}}\;dx\leq \int_{\mathcal{D}} f_{n}G_{k}(u_{n,k})\;dx.$$
We apply the standard technique from \cite{Dicastro} to obtain $u_{n,k} \in L^{\infty}(\mathcal{D})$. Moreover, the $L^{\infty}$ estimate is independent of $k$, so for sufficiently large $k$ and fixed $n$, $u_n \in W_{0}^{1, (p_{j})}(\mathcal{D}) \cap L^{\infty}(\mathcal{D})$ solves \eqref{(2.1)}.
\par Choosing $\varphi = u_n^{-} = \min\{u_n, 0\}$ in \eqref{(2.1)} and applying \eqref{(1.2)}, we observe that $\frac{f_n}{\left(u_{n}+\frac{1}{n}\right)^{\gamma}}$ is nonnegative. Thus,  
\begin{equation*}
\alpha\sum_{j\in\mathcal{E}}\int_{\mathcal{D}}\vert \partial_i u_{n}^{-}\vert^{p_{j}}\;dx\leq\int_{\mathcal{D}} \frac{f_{n} \;dx}{\left(u_{n}+\frac{1}{n}\right)^{\gamma}} u_{n}^{-}\leq 0.
\end{equation*}
This implies $u_n^- = 0$ a.e. in $\mathcal{D}$, leading to $u_n \geq 0$ a.e. in $\mathcal{D}$.
\end{proof}
\par We show in the following Lemma that $u_{n}$ remains strictly positive in $\mathcal{D}$.
\begin{lemma} 
\label{LL3}
Let $u_{n}$ solve \eqref{(2.1)}. Then, there exists a constant $\tilde{C}>0$, independent of $n$, such that $u_{n} \geq \tilde{C}$ a.e. in $\mathcal{D}$ for all $ n \in \mathbb{N}$.
\end{lemma}
\begin{proof}
For $s \geq 0$, consider the function $\Theta_{\sigma}(s)$ defined as follows
$$\Theta_{\sigma}(s)=\left\{\begin{array}{ll}1&\hbox{if}\;0\leq s\leq1, \\
\frac{1}{\sigma}(1+\sigma-s) & \hbox{if}\; 1< s\leq\sigma+1,\\ 
0 & \hbox{if}\; s>\sigma+1,
\end{array}
\right.$$
Using $\Theta_{\sigma}(s)\phi$ as a test function in \eqref{(2.1)}, where $\phi \in W_{0}^{1,(p_{j})}(\mathcal{D}) \cap L^{\infty}(\mathcal{D})$ with $\phi \geq 0$, we obtain
\begin{align*}
&\sum_{j\in\mathcal{E}}\int_{\mathcal{D}}(b(x)+u_{n}^{q})\vert \partial_{j}u_{n}\vert^{p_{j}-2}\partial_{j}u_{n}\partial_{j}\phi\Theta_{\sigma}(u_{n})\;dx\\
&\quad=\frac{1}{\sigma}\int_{\mathcal{D}\cap \{1\leq u_{n}\leq \sigma+1\}}(b(x)+u_{n}^{q})\vert \partial_{j}u_{n}\vert^{p_{j}}\phi \;dx +\int_{\mathcal{D}}\frac{f_{n}}{\left(u_{n}+\frac{1}{n}\right)^{\gamma}}\Theta_{\sigma}(u_{n})\phi \;dx.
\end{align*}
Omitting the non-negative term on the right-hand side and letting $\sigma \to 0$, we derive
\begin{align*}
&\sum_{j\in\mathcal{E}}\int_{\mathcal{D}}(b(x)+u_{n}^{q})\vert \partial_{j}u_{n}\vert^{p_{j}-2}\partial_{j}u_{n}\partial_{j}\phi\chi_{\{0\leq u_{n}\leq1\}}\;dx\geq \int_{\mathcal{D}}\frac{f_{n} \;dx}{\left(u_{n}+\frac{1}{n}\right)^{\gamma}}\phi\chi_{\{0\leq u_{n}\leq1\}}.
\end{align*}
Using the fact that $\frac{f_{n}}{\left(u_{n}+\frac{1}{n}\right)^{\gamma}}\geq \frac{T_{1}(f)}{2^{\gamma}}$ on the set $\{0\leq u_{n}\leq1\}$, we have
\begin{align*}
&\sum_{j\in\mathcal{E}}\int_{\mathcal{D}}(b(x)+T_{1}(u_{n})^{q})\vert \partial_{j}T_{1}(u_{n})\vert^{p_{j}-2}\partial_{j}T_{1}(u_{n})\partial_{j}\phi \;dx\geq \frac{1}{2^{\gamma}}\int_{\mathcal{D}}T_{1}(f)\phi\chi_{\{0\leq u_{n}\leq1\}}\;dx.
\end{align*}
Since $b(x)+T_{1}(u_{n})^{q}\leq \beta+1$, we get
\begin{align*}
&(\beta+1)\sum_{j\in\mathcal{E}}\int_{\mathcal{D}}\vert \partial_{j}T_{1}(u_{n})\vert^{p_{j}-2}\partial_{j}T_{1}(u_{n})\partial_{j}\phi \;dx\geq 0.
\end{align*}
This yields that $T_{1}(u_{n})$ is a super-solution solution of the following problem
\begin{equation*}
\left\{\begin{array}{ll} \displaystyle -\sum_{j\in\mathcal{E}}\partial_{j}\left(\vert \partial_{j} Z \vert^{p_{j}-2} \partial_{j} Z\right)= 0 & \hbox{in}\;\;\mathcal{D}, \\
Z=0 & \hbox{on}\;\; \partial\mathcal{D},
 \end{array}
 \right.
\end{equation*}
It follows from Theorem \ref{T1} and Corollary \ref{C1} that $Z > 0$ in $\mathcal{D}$, and therefore, $T_{1}(u_{n}) > 0$ in $\mathcal{D}$. By the strict monotonicity of $ T_{1}$, there exists $\tilde{C} > 0$ such that $u_{n} \geq \tilde{C}$ in $\mathcal{D} $ for all $n \in \mathbb{N}$.
\end{proof}
\subsection{A priori estimates}
In this subsection, $C$ denotes a positive constant independent of $n$, which may vary between lines.
\begin{lemma}
\label{LL1}
Let $u_n$ be a solution of \eqref{(2.1)}, with $f \in L^{1}(\mathcal{D})$. Then
\begin{itemize}
\item[(1)] If $0 < \gamma < 1$:
\begin{itemize}
\item[(i)]  $u_{n}$ is bounded in $W_{0}^{1, (p_{j})}(\mathcal{D})$ for $q> 1-\gamma$.
\item[(ii)] $u_{n}$ is bounded in $W_{0}^{1,(p_{j})}(\mathcal{D})$ for $q=1-\gamma$.
\item[(iii)] $u_{n}$ is bounded in  $W_{0}^{1, (\eta_{j})}(\mathcal{D})$ for $q<1-\gamma$, with 
$$\eta_{j}=\frac{N(q+\gamma+\overline{p}-1)}{[N-(1-(q+\gamma))]\overline{p}}p_{j}\quad \forall\; j\in\mathcal{E}.$$
\end{itemize}
\item[(2)] If $\gamma>1$ and $q\geq0$, then $u_{n}$ is bounded in $W_{0}^{1, (p_{j})}(\mathcal{D})\cap L^{r(\gamma)}(\mathcal{D})$, with
$$r(\gamma)=\frac{N(q+\gamma-1+\overline{p})}{N-\overline{p}}.$$
\item[(3)] If $\gamma=1$ and $q\geq0$, then $u_{n}$ is bounded in $W_{0}^{1, p}(\mathcal{D})\cap L^{r(1)}(\mathcal{D})$.
\end{itemize}
\end{lemma}
\begin{proof} \textit{Proof of (1)-(i):}
Taking $u_n^{\gamma}\left(1-(1+u_{n})\right)^{1-(q+\gamma)}$ as a test function in \eqref{(2.1)}, leads to
\begin{align*}
&\gamma\sum_{j\in\mathcal{E}}\int_{\mathcal{D}}\left(b(x)+u_{n}^{q}\right)\vert \partial_{j} u_{n}\vert^{p_{j}} u_n^{\gamma -1} \left( 1-(1+u_n)^{1-(q+\gamma)}\right)\;dx \nonumber \\
&\qquad+(q+\gamma-1)\sum_{j\in\mathcal{E}}\int_{\mathcal{D}}\left(b(x)+u_{n}^{q}\right)u_n^{\gamma}\frac{\vert \partial_{j} u_{n}\vert^{p_{j}} }{(1+u_n)^{q+\gamma}}\;dx  \nonumber \\
&\quad=\int_{\mathcal{D}}\frac{f_{n}}{\left( u_{n}+\frac{1}{n}\right)^{\gamma}}u_n^{\gamma}\left(1-(1+u_{n})\right)^{1-(q+\gamma)}\;dx.
\end{align*}
By neglecting the first positive term on the left-hand side and using the fact that 
$$\frac{f_{n}}{\left( u_{n}+\frac{1}{n}\right)^{\gamma}}u_n^{\gamma}\left(1-(1+u_{n})\right)^{1-(q+\gamma)}\leq f,$$
we obtain
\begin{equation} 
\label{15}
\int_{\mathcal{D}}\left(b(x)u_{n}^{\gamma}+u_{n}^{q+\gamma}\right)\frac{\vert \partial_{j} u_{n}\vert^{p_{j}} }{(1+u_n)^{q+\gamma}}\;dx \leq C\Vert f\Vert_{L^{1}(\mathcal{D})},\quad \forall\; j\in\mathcal{E}.
\end{equation}
By working in $\{u_{n}\geq 1\}$, we have for all $j\in\mathcal{E}$
\begin{align}
\int_{\mathcal{D}}\left(\alpha u_{n}^{\gamma}+u_{n}^{q+\gamma}\right) \frac{\vert \partial_{j} u_{n}\vert^{p_{j}}}{\left(1+u_{n}\right)^{q+\gamma}}\;dx
 &\geq \frac{\min(1,\alpha)}{2^{q+\gamma-1}}\int_{\left\{u_{n} \geq 1\right\}} \vert \partial_{j}u_{n}\vert^{p_{j}}\;dx. \label{17}
\end{align}
Then it follows from \eqref{15} and \eqref{17} that
\begin{equation}
\int_{\left\{u_{n} \geq 1\right\}} \vert \partial_{j}u_{n}\vert^{p_{j}}\;dx \leq C\Vert f\Vert_{L^{1}(\mathcal{D})},\quad \forall\; j\in\mathcal{E}. \label{a}
\end{equation}
After that, we will be testing \eqref{(2.1)} by $T_{k}^{\gamma}\left(u_{n}\right)$ for all $k>0$, in order to obtain
\begin{align*}
& \gamma\sum_{j\in\mathcal{E}}\int_{\mathcal{D}}\left(b(x)+u_{n}^{q}\right)\frac{\vert \partial_{j} T_{k}(u_{n})\vert^{p_{j}}}{[T_{k}(u_{n})]^{1-\gamma}}\;dx 
= \int_{\mathcal{D}} \frac{f_{n}[T_{k}(u_{n})]^{\gamma}}{\left(u_{n}+\frac{1}{n}\right)^{\gamma}} \;dx
\end{align*}
For the left hand side, we dropping the positive term and use \eqref{(1.2)}, 
as for the right hand side, using that $T_k(u_n)\leq u_n$, we have
\begin{equation}
\label{aa}
\int_{\mathcal{D}}\frac{\vert \partial_{j} T_{k}(u_{n})\vert^{p_{j}}}{[T_{k}(u_{n})]^{1-\gamma}}\;dx\leq C\Vert f\Vert_{L^{1}(\mathcal{D})},\quad \forall\; j\in\mathcal{E}.
\end{equation}
On the other hand, knowing that $T_k (u_n) \leq k $ and according to \eqref{aa}, we have 
\begin{align}
\int_{\mathcal{D}}\vert \partial_{j}T_{k}(u_{n})\vert^{p_{j}}\;dx &=\int_{\mathcal{D}} \frac{\vert \partial_{j} T_{k}(u_{n})\vert^{p_{j}}}{[T_k(u_n)]^{1-\gamma }} [T_k(u_n)]^{1-\gamma }\;dx\nonumber\\
  &\leq k^{1-\gamma}C\Vert f\Vert_{L^{1}(\mathcal{D})},\quad \forall\; j\in\mathcal{E}. \label{b}
\end{align}
According to \eqref{a} and \eqref{b}, it follows that
$$
\int_{\mathcal{D}}\vert \partial_{j} u_{n}\vert^{p_{j}}\;dx \leq C ,\quad \forall\; j\in\mathcal{E}.
$$
Meaning, we were able to show the boundedness of $u_{n}$ in $W_{0}^{1, (p_{j})}(\mathcal{D})$ as an outcome of the previous estimate.
\par \textit{Proof of (1)-(ii):} Taking $u_{n}^{\gamma}$ as a test function in \eqref{(2.1)}, we arrive at
\begin{align*}
& \sum_{j\in\mathcal{E}}\int_{\mathcal{D}}\left(b(x)u_{n}^{\gamma-1}+u_{n}^{q+\gamma-1}\right)\vert \partial_{j} u_{n}\vert^{p_{j}}\;dx= \int_{\mathcal{D}} \frac{f_{n}}{\left(u_{n}+\frac{1}{n}\right)^{\gamma}}u_{n}^{\gamma}\;dx.
\end{align*}    
Since $b(x)u_{n}^{\gamma-1}\geq \alpha u_{n}^{\gamma-1}\geq0$ and $\frac{f_{n}}{\left(u_{n}+\frac{1}{n}\right)^{\gamma}}u_{n}^{\gamma}\leq f$, we obtain
\begin{equation}
\int_{\mathcal{D}}\frac{\vert \partial_{j} u_{n}\vert^{p_{j}}}{u_{n}^{1-q-\gamma}}\;dx \leq \Vert f\Vert_{L^{1}(\mathcal{D})}, \quad \forall\; j\in\mathcal{E}. \label{18}
\end{equation}
 If $q=1-\gamma$, then from \eqref{18} we get $u_{n}$ is bounded in $W_{0}^{1,(p_{j})}(\mathcal{D})$. 
\par \textit{Proof of (1)-(iii):} If $q<1-\gamma$, we select $ \eta_{j},$ such that $ 1< \eta_{j}\leq p_{j} $ for all $j\in\mathcal{E}$ and employing H\"older's inequality with exponents $\left(\frac{p_{j}}{\eta_{j}},\left(\frac{p_{j}}{\eta_{j}}\right)'\right)$ along with \eqref{18}, we derive
\begin{align}
\label{(8.6)}
\int_{\mathcal{D}}\vert \partial_{j}u_{n}\vert^{\eta_{j}}\;dx&=\int_{\mathcal{D}}\frac{\vert \partial_{j}u_{n}\big\vert^{\eta_{j}}}{u_{n}^{(1-q-\gamma)\frac{ \eta_{j}}{p_{j}}}}u_{n}^{(1-q-\gamma)\frac{ \eta_{j}}{p_{j}}}\;dx\nonumber\\
&\leq \left(\int_{\mathcal{D}}\frac{\vert \partial_{j}u_{n}\vert^{p_{j}} }{u_{n}^{1-q-\gamma}}\;dx\right)^{\frac{\eta_{j}}{p_{j}}}\left(\int_{\mathcal{D}}u_{n}^{(1-q-\gamma)\frac{ {\eta_{j}}}{p_{j}- \eta_{j}}}\;dx\right)^{\frac{p_{j}-\eta_{j}}{p_{j}}}\nonumber\\
&\leq C\left(\int_{\mathcal{D}}u_{n}^{(1-q-\gamma)\frac{ {\eta_{j}}}{p_{j}- \eta_{j}}}\;dx\right)^{\frac{p_{j}-\eta_{j}}{p_{j}}}.
\end{align}
We set $ \eta_{j}=\varrho p_{j} $, where $\varrho\in [0,1)$, yielding
\begin{align}
\label{(8.7)}
(1-q-\gamma)\frac{ {\eta_{j}}}{p_{j}- \eta_{j}}=(1-q-\gamma)\frac{\varrho}{1-\varrho}=\frac{\varrho\overline{p}N}{N-\varrho\overline{p}}=\overline{\eta}^{*}.
\end{align}
This inequality, along with the condition $\overline{p}<N$, indicates that
\begin{align}
\label{(08.7)}
\varrho=\frac{N(\overline{p}+q+\gamma-1)}{\overline{p}(N+q+\gamma-1)}<1.
\end{align}
By \eqref{(8.6)} and \eqref{(8.7)}, we have
\begin{align}
\label{(008.7)}
\int_{\mathcal{D}}\vert \partial_{j}u_{n}\vert^{\eta_{j}}\;dx \leq C\left(\int_{\mathcal{D}}u_{n}^{\overline{\eta}^{*}}\;dx\right)^{1-\varrho}\;dx,
\end{align}
hence,
\begin{align*}
\prod_{j\in\mathcal{E}}\left(\int_{\mathcal{D}}\vert \partial_{j}u_{n}\vert^{\eta_{j}}\;dx\right)^{\frac{1}{N\eta_{j}}}&\leq C\prod_{j\in\mathcal{E}}\left(\int_{\mathcal{D}}u_{n}^{\overline{\eta}^{*}}\;dx\right)^{\frac{1-\varrho}{\eta_{j}N}}\\
&=\left(\int_{\mathcal{D}}u_{n}^{\overline{\eta}^{*}}\;dx\right)^{\frac{1-\varrho}{\overline{\eta}}}.
\end{align*}
Then, by applying \eqref{(4)} from Lemma \ref{L1} with $w = u_{n}$ and $\tau = \overline{\eta}^*$, one obtains
\begin{align*}
\Vert u_{n}\Vert_{L^{\overline{\eta}^{*}}(\mathcal{D})}\leq C \Vert u_{n}\Vert_{L^{\overline{\eta}^{*}}(\mathcal{D})}^{\frac{(1-\varrho)\overline{\eta}^{*}}{\overline{\eta}}}.
\end{align*}
Since $1>\frac{(1-\varrho)\overline{\eta}^{*}}{\overline{\eta}}$ (since $\overline{p}<N$), we have $u_{n}$ in bounded in $L^{\overline{\eta}^{*}}(\mathcal{D})$. Hence, from \eqref{(08.7)} and \eqref{(008.7)}, we deduce that  $u_{n}$ is bounded in $W_{0}^{1,(\eta_{j})}(\mathcal{D})$.
 \par \textit{Proof of (2):}
We choose $u_{n}^{\gamma} $ as a test function in \eqref{(2.1)} and apply \eqref{(1.2)}, which gives us
\begin{align}
\label{E5}
\alpha\gamma\sum_{j\in\mathcal{E}}\int_{\mathcal{D}}u_{n}^{\gamma-1}\vert \partial_{j}u_{n}\vert^{p_{j}}\;dx+\sum_{j\in\mathcal{E}}\int_{\mathcal{D}}u_{n}^{q+\gamma-1}\vert \partial_{j}u_{n}\vert^{p_{j}}\;dx\leq  \Vert f\Vert_{L^{1}(\mathcal{D})}.
\end{align}
Dropping the positive term in \eqref{E5}, we get
$$\int_{\mathcal{D}}u_{n}^{q+\gamma-1}\vert \partial_{j}u_{n}\vert^{p_{j}}\;dx\leq  \Vert f\Vert_{L^{1}(\mathcal{D})}\quad \forall\; j\in\mathcal{E},$$
then 
\begin{align}
\label{E6}
\prod_{j\in\mathcal{E}}\left(\int_{\mathcal{D}}u_{n}^{q+\gamma-1}\vert \partial_{j}u_{n}\vert^{p_{j}}\;dx\right)^{\frac{1}{p_{j}}}\leq \Vert f\Vert_{L^{1}(\mathcal{D})}^{\frac{N}{\overline{p}}}
\end{align}
Using \eqref{E6} and \eqref{(9)}, we obtain
\begin{align}
\label{E7}
\left(\int_{\mathcal{D}}u_{n}^{r}\;dx\right)^{\frac{N}{\overline{p}}-1}\leq C,
\end{align}
where
\begin{equation*}
\left\{\begin{array}{ll}r=\dfrac{1+t_{j}}{b_{j}(N-1)-1+\frac{1}{p_{j}}}, & t_{j}\geq0,\;\; \forall\; j\in\mathcal{E}, \\
t_{j}p_{j}=q+\gamma-1, & \forall\; j\in\mathcal{E}, \\
\sum_{j\in\mathcal{E}}b_{j}=1, & b_{j}\geq0,\;\; \forall\; j\in\mathcal{E}.
 \end{array}
 \right.
\end{equation*}
This implies that
\begin{align}
\label{E8}
r=\frac{N(q+\gamma-1+\overline{p})}{N-\overline{p}}
\end{align}
From \eqref{E7} and \eqref{E8}, it follows the boundedness of $u_{n}$ in $L^{r}(\mathcal{D})$. 
\par Therefore, by \eqref{E5}, we obtain
$$\int_{\mathcal{D}}u_{n}^{\gamma-1}\vert \partial_{j}u_{n}\vert^{p_{j}}\;dx\leq \frac{\Vert f\Vert_{L^{1}(\mathcal{D})}}{\gamma\alpha},\quad \forall\; j\in\mathcal{E}.$$
Using Lemma \ref{LL3}, we get
$$\int_{\mathcal{D}}\vert \partial_{j}u_{n}\vert^{p_{j}}\;dx\leq \frac{\Vert f\Vert_{L^{1}(\mathcal{D})}}{\tilde{C}^{\gamma-1}\gamma\alpha},\quad \forall\; j\in\mathcal{E},$$
this implies that $u_{n}$ is bounded in $W_{0}^{1,(p_{j})}(\mathcal{D})$.
\par \textit{Proof of (3):} The proof of (3) is similar to the one of (2) with set $\gamma=1$.
\end{proof}
\begin{lemma} 
\label{LL2} 
Let $u_{n}$ be a solution to \eqref{(2.1)}. Assume that $q\geq0$ and $f \in L^{1}(\mathcal{D})$, then 
\begin{itemize}
\item[(1)] If $0<\gamma\leq1$, then $u_{n}^{q}\left|\partial_{j} u_{n}\right|^{p_{j}-1}$ is bounded in $L^{\delta_{j}}(\mathcal{D})$ for every $1<\delta_{j}<p_{j}^{\prime}$ and for all $j\in\mathcal{E}$.
\item[(2)] If $\gamma>1$ and $q<(\gamma-1)(p_{j}-1)$ for all $j\in\mathcal{E}$, then $u_{n}^{q}\left|\partial_{j} u_{n}\right|^{p_{j}-1}$ is bounded in $L^{p'_{j}}(\mathcal{D})$ for all $j\in\mathcal{E}$.
\end{itemize}
\end{lemma}
\begin{proof} \textit{Proof of (1):}
Testing \eqref{(2.1)} with 
$$(T_{1}(u_{n}))^{\gamma}\left(1-\frac{1}{\left(1+u_{n}\right)^{\lambda-1}}\right),\quad \mbox{with}\;  \lambda>1,$$
we get
\begin{align*}
&\gamma \sum_{j\in\mathcal{E}}\int_{\mathcal{D}}\left(T_{1}(u_{n})\right)^{\gamma-1}\left(1-\frac{1}{\left(1+u_{n}\right)^{\lambda-1}}\right)\left(b(x)+u_{n}^{q}\right)\left\vert \partial_{j} T_{1}\left(u_{n}\right)\right\vert^{p_{j}}\;dx\\
&\qquad +(\lambda-1)\sum_{j\in\mathcal{E}} \int_{\mathcal{D}}\left(T_{1}(u_{n})\right)^{\gamma}\left(b(x)+u_{n}^{q}\right) \frac{\left\vert \partial_{j} u_{n}\right\vert^{p_{j}}}{\left(1+u_{n}\right)^{\lambda}}\;dx\\
&\quad=\int_{\mathcal{D}} \frac{f_{n}}{\left(u_{n}+\frac{1}{n}\right)^{\gamma}}(T_{1}(u_{n}))^{\gamma}\left(1-\frac{1}{\left(1+u_{n}\right)^{\lambda-1}}\right)\;dx.
\end{align*}
Neglecting the positive terms on the left-hand side, use \eqref{(1.2)} and the fact that
$$\frac{b(x)+u_{n}^{q}}{(1+u_{n})^{\lambda}}\geq \frac{C_{\alpha}}{(1+u_{n})^{\lambda-q}},\quad (T_{1}(u_{n}))^{\gamma}\left(1-\frac{1}{\left(1+u_{n}\right)^{\lambda-1}}\right)\leq u_{n}^{\gamma},$$
we obtain
$$C_{\alpha}(\lambda-1) \int_{\mathcal{D}}T_{1}(u_{n})^{\gamma}\frac{\left\vert \partial_{j}u_{n}\right\vert^{p_{j}}}{\left(1+u_{n}\right)^{\lambda-q}} \;dx\leq \Vert f\Vert_{L^{1}(\mathcal{D})},\quad \forall\; j\in\mathcal{E}.$$
Now working on $\{u_{n} \geq 1\}$, we have
$$\int_{\{u_{n} \geq 1\}}\frac{\left\vert \partial_{j}u_{n}\right\vert^{p_{j}}}{\left(1+u_{n}\right)^{\lambda-q}} \;dx\leq C,\quad \forall\; j\in\mathcal{E}.$$
Consequently, using \eqref{b} and the above estimate, we reach
\begin{align}
\label{98} 
\int_{\mathcal{D}}\frac{\left\vert \partial_{j}u_{n}\right\vert^{p_{j}}}{\left(1+u_{n}\right)^{\lambda-q}}\;dx=\int_{\{u_{n} \geq 1\}}\frac{\left\vert \partial_{j}u_{n}\right\vert^{p_{j}}}{\left(1+u_{n}\right)^{\lambda-q}}\;dx+\int_{\{u_{n} < 1\}}\frac{\left\vert \partial_{j}T_{1}(u_{n})\right\vert^{p_{j}}}{\left(1+u_{n}\right)^{\lambda-q}}\;dx\leq C.
\end{align}
Next, for $1 < \rho_{j} < p'_{j} = \frac{p_{j}}{p_{j} - 1}$, applying H\"older's inequality and \eqref{98}, we obtain the following for all $j\in\mathcal{E}$
\begin{align}
\int_{\mathcal{D}} u_{n}^{q \rho_{j}(p_{j}-1)}\left|\partial_{j} u_{n}\right|^{\rho_{j}(p_{j}-1)}\;dx &\leq    \int_{\mathcal{D}}\left(1+u_{n}\right)^{q \rho_{j}(p_{j}-1)}\left|\partial_{j} u_{n}\right|^{\rho_{j}(p_{j}-1)}\;dx \nonumber \\
&=\int_{\mathcal{D}} \frac{\left|\partial_{j} u_{n}\right|^{\rho_{j}(p_{j}-1)}}{\left(1+u_{n}\right)^{\frac{(\lambda-q) \rho_{j}(p_{j}-1)}{p_{j}}}}\left(1+u_{n}\right)^{\frac{(\lambda-q) \rho_{j}(p_{j}-1)}{p_{j}}+\rho_{j} q(p_{j}-1)}\;dx \nonumber \\
&=\int_{\mathcal{D}} \frac{\left|\partial_{j} u_{n}\right|^{\rho_{j}(p_{j}-1)}}{\left(1+u_{n}\right)^{\frac{(\lambda-q) \rho_{j}(p_{j}-1)}{p_{j}}}}\left(1+u_{n}\; \right)^{\frac{\rho_{j}(p_{j}-1)(\lambda-q+p_{j} q)}{p_{j}}} \nonumber \;dx\\
&\leq \left( \int_{\mathcal{D}} \frac{\left|\partial_{j} u_{n}\right|^{p_{j}}}{\left(1+u_{n}\right)^{\lambda-q}}\;dx\right)^{\frac{\delta_{j}}{p'_{j}}} \left( \int_{\mathcal{D}}\left(1+u_{n}\right)^{\frac{\rho_{j}(p_{j}-1)(\lambda-q+p_{j} q)}{p_{j}-\rho_{j}(p_{j}-1)}}\;dx\right)^{\frac{p_{j}-\rho_{j}(p_{j}-1)}{p_{j}}} \nonumber \\
& \leq C\left( \int_{\mathcal{D}} (1+u_{n})^{\frac{\rho_{j}(p_{j}-1)(\lambda-q+p_{j} q)}{p_{j}-\rho_{j}(p_{j}-1)}}\;dx\right)^{\frac{p_{j}-\rho_{j}(p_{j}-1)}{p_{j}}},  \nonumber
\end{align}
which implies for all $j\in\mathcal{E}$
\begin{equation} 
\label{99}
\int_{\mathcal{D}} u_{n}^{q \rho_{j}(p_{j}-1)}\left|\partial_{j} u_{n}\right|^{\rho_{j}(p_{j}-1)} \;dx\leq C\left( \int_{\mathcal{D}} (1+u_{n})^{\frac{\rho_{j}(p_{j}-1)(\lambda-q+p_{j} q)}{p_{j}-\rho_{j}(p_{j}-1)}}\;dx\right)^{\frac{p_{j}-\rho_{j}(p_{j}-1)}{p_{j}}}.
\end{equation}
We define $\rho_{j}=\theta\frac{p_{j}}{p_{j}-1}$ with $\theta\in (0,1)$, which means that for all $j\in\mathcal{E}$
\begin{align}
\label{E10}
\frac{p_{j}-\rho_{j}(p_{j}-1)}{p_{j}}=1-\theta,\quad \mbox{and}\quad \frac{\rho_{j}(p_{j}-1)(\lambda-q+p_{j} q)}{p_{j}-\rho_{j}(p_{j}-1)}=\frac{\theta}{1-\theta}\left[\lambda+q(p_{j}-1)\right].
\end{align}
From \eqref{99} and \eqref{E10}, we obtain for all $j\in\mathcal{E}$
\begin{equation} 
\label{E11}
\int_{\mathcal{D}} u_{n}^{q \rho_{j}(p_{j}-1)}\left|\partial_{j} u_{n}\right|^{\rho_{j}(p_{j}-1)} \;dx\leq C\left( \int_{\mathcal{D}}(1+ u_{n})^{\frac{\theta}{1-\theta}\left[\lambda+q(p_{j}-1)\right]}\;dx\right)^{1-\theta}.
\end{equation}
Hence 
$$\prod_{j\in\mathcal{E}}\left(\int_{\mathcal{D}} u_{n}^{q \rho_{j}(p_{j}-1)}\left|\partial_{j} u_{n}\right|^{\rho_{j}(p_{j}-1)}\;dx\right)^{\frac{1}{\theta p_{j}}} \leq C\prod_{j\in\mathcal{E}}\left( \int_{\mathcal{D}} (1+u_{n})^{\frac{\theta}{1-\theta}\left[\lambda+q(p_{j}-1)\right]}\;dx\right)^{\frac{1-\theta}{\theta p_{j}}}.$$
Using the inequality anisotropic \eqref{(9)} with $\delta_{j}=\rho_{j}(p_{j}-1)\leq p_{j}$ (since $\theta<1$), $t_{j}=q\geq0$, $\overline{\delta}=\theta\overline{p}$, we have
\begin{equation}
\label{E12}
\left(\int_{\mathcal{D}}u_{n}^{r}\;dx\right)^{\frac{N}{\theta\overline{p}}-1}\leq C\prod_{j\in\mathcal{E}}\left( \int_{\mathcal{D}} (1+u_{n})^{\frac{\theta}{1-\theta}\left[\lambda+q(p_{j}-1)\right]}\;dx\right)^{\frac{1-\theta}{\theta p_{j}}}.
\end{equation}
Since we require that $r=\frac{\theta}{1-\theta}\left[\lambda+q(p_{j}-1)\right]$ in \eqref{E12} for all $j\in\mathcal{E}$, we must solve the following system
\begin{empheq}[left=\empheqlbrace]{align}
&r=\dfrac{1+q}{b_{j}(N-1)-1+\frac{1}{\theta p_{j}}},  \;\; \forall\; j\in\mathcal{E}, \label{S1} \\ 
&r=\frac{\theta}{1-\theta}\left[\lambda+q(p_{j}-1)\right],  \forall\; j\in\mathcal{E}, \label{S2} \\
&\sum_{j\in\mathcal{E}}b_{j}=1,  b_{j}\geq0,\;\; \forall\; j\in\mathcal{E}. \label{S3}
\end{empheq}
From \eqref{S1} and \eqref{S3}, we obtain
\begin{align}
\label{S4}
r=\frac{N(1+q)\theta \overline{p}}{N-\theta\overline{p}},
\end{align}
and by \eqref{S2}, we get 
\begin{align}
\label{S5}
\theta=\frac{N(\overline{p}-\lambda+q)}{\overline{p}[N(1+q)-q\overline{p}-\lambda+q]}.
\end{align}
Since $q\geq0$ and $\lambda>1$ we obtain $0<\theta<1$. Combining \eqref{S4} and \eqref{S5}, we get
\begin{align}
\label{S6}
r=\frac{N(\overline{p}-\lambda+q)}{N-\overline{p}}.
\end{align}
By \eqref{E12} and \eqref{S6}, we find
\begin{align}
\label{S7}
\left(\int_{\mathcal{D}}u_{n}^{r}\;dx\right)^{\frac{N}{\theta\overline{p}}-1}\leq C\left(\int_{\mathcal{D}}u_{n}^{r}\;dx\right)^{\frac{(1-\theta)N}{\theta \overline{p}}},
\end{align}
and the assumption $\overline{p}<N$ implies that $\frac{N}{\theta\overline{p}}-1>\frac{(1-\theta)N}{\theta \overline{p}}$. Thus, from \eqref{S7}, we obtain the boundedness of $u_{n}$ in $L^{r}(\mathcal{D})$, which, according to \eqref{E11}, implies that
\begin{align}
\label{S8}
\int_{\mathcal{D}} u_{n}^{q \rho_{j}(p_{j}-1)}\left|\partial_{j} u_{n}\right|^{\rho_{j}(p_{j}-1)}\;dx \leq C,\quad \forall\; j\in\mathcal{E}.
\end{align}
As $p_{j}>2$ and $\rho_{j}>1$ for all $j\in\mathcal{E}$, according to \eqref{S8}, and using H\"older's inequality, we derive
\begin{align*}
\int_{\mathcal{D}} u_{n}^{q \rho_{j}}\left|\partial_{j} u_{n}\right|^{\rho_{j}(p_{j}-1)}\;dx &=  \int_{\left\{u_{n}<1\right\}} u_{n}^{q \rho_{j}}\left|\partial_{j} u_{n}\right|^{\rho_{j}(p_{j}-1)}\;dx + \int_{\left\{u_{n} \geq 1\right\}} u_{n}^{q \rho_{j}}\left|\partial_{j} u_{n}\right|^{\rho_{j}(p_{j}-1)}\;dx \\
&\leq \int_{\mathcal{D}} \left|\partial_{j} T_{1}(u_{n})\right|^{\rho_{j}(p_{j}-1)}\;dx + \int_{\mathcal{D}} u_{n}^{q \rho_{j}(p_{j}-1)}\left|\partial_{j} u_{n}\right|^{\rho_{j}(p_{j}-1)}\;dx \\
&\leq C\left(\int_{\mathcal{D}} \left|\partial_{j} T_{1}(u_{n})\right|^{p_{j}}\;dx\right)^{\frac{\rho_{j}(p_{j}-1)}{p_{j}}}+C\\
&\leq C.
\end{align*}
Then $u_{n}^{q}\left|\partial_{j} u_{n}\right|^{p_{j}-1}$ is bounded in $L^{\rho_{j}}(\mathcal{D})$ for every $1<\rho_{j}<p'_{j}$ and for all $j\in\mathcal{E}$.
\par \textit{Proof of (1):} We consider $u_{n}^{\gamma}$ as a test function in \eqref{(2.1)}, applying \eqref{(1.2)}, $\frac{u_{n}^{\gamma}}{\left(u_{n}+\frac{1}{n}\right)^{\gamma}}\leq 1$, $f_{n}\leq f$, and ignoring the positive term, we get
$$
\sum_{j\in\mathcal{E}}\int_{\mathcal{D}}u_{n}^{q+\gamma-1}\vert \partial_{j}u_{n}\vert^{p_{j}}\;dx\leq\Vert f\Vert_{L^{1}(\mathcal{D})}.
$$
Therefore, we have
\begin{align}
\label{S9}
\int_{\{u_{n}\geq1\}}u_{n}^{q+\gamma-1}\vert \partial_{j}u_{n}\vert^{p_{j}}\;dx\leq C,\quad \forall\; j\in\mathcal{E}.
\end{align}
As $q\leq (p_{j}-1)(\gamma-1)$ for all $j\in\mathcal{E}$, it follows that
\begin{align}
\label{S10}
qp'_{j}\leq q+\gamma-1.
\end{align}
Using \eqref{S9} and \eqref{S10}, we deduce for all $j\in\mathcal{E}$
\begin{align*}
\int_{\mathcal{D}}\left\vert u_{n}^{q} \vert\partial_{j}u_{n}\vert^{p_{j}-2}\partial_{j}u_{n}\right\vert^{p'_{j}}\;dx&=\int_{\mathcal{D}}u_{n}^{qp'_{j}}\vert \partial_{j}u_{n}\vert^{p_{j}}\;dx\\
&=\int_{\{u_{n}> 1\}}u_{n}^{qp'_{j}}\vert \partial_{j}u_{n}\vert^{p_{j}}\;dx+\int_{\{u_{n}\leq 1\}}u_{n}^{qp'_{j}}\vert \partial_{j}u_{n}\vert^{p_{j}}\;dx\\
&\leq \int_{\{u_{n}>1\}}u_{n}^{q+\gamma-1}\vert \partial_{j}u_{n}\vert^{p_{j}}\;dx+\int_{\mathcal{D}}\vert \partial_{j}T_{1}(u_{n})\vert^{p_{j}}\;dx\\
&\leq C.
\end{align*}
Then the last inequality ensure that $u_{n}^{q}\left|\partial_{j} u_{n}\right|^{p_{j}-1}$ is bounded in $L^{p'_{j}}(\mathcal{D})$ for all $j\in\mathcal{E}$.
\end{proof}
\subsection{Passing to the limit}
Because the proof of the cases $(\textbf{B}_{\gamma,q})$ and $(\textbf{C}_{\gamma,q})$ in Theorem \ref{theo1} are similar to that of the case $(\textbf{A}_{\gamma,q})$, we restrict to that of the case (1). By applying Lemma \ref{LL1}, we conclude that $u_{n}$ is bounded in $W_{0}^{1,(p_{j})}(\mathcal{D})$ (or equivalently in $W_{0}^{1,(\eta_{j})}(\mathcal{D})$). It follows that there exists a function $u$ such that
\begin{align}
\label{P1}
u_{n}\rightharpoonup u\quad \mbox{weakly in}\;\; W_{0}^{1,(p_{j})}(\mathcal{D}) \; (\mbox{or in}\;\; W_{0}^{1,(\eta_{j})}(\mathcal{D}))\;\; \mbox{a.e. in}\; \mathcal{D}.
\end{align}
Furthermore, according to \cite{amine}, it follows that
\begin{align}
\label{P2}
\partial_{j}u_{n}\rightharpoonup \partial_{j}u\quad \mbox{a.e. in}\; \mathcal{D}.
\end{align}
For the first term, by \eqref{P2}, we observe that $b(x)\vert \partial_{j}u_{n}\vert^{p_{j}-2}\partial_{j}u_{n}$ converge to $b(x)\vert \partial_{j}u\vert^{p_{j}-2}\partial_{j}u$ almost everywhere in $\mathcal{D}$. Furthermore, since $u_{n}$ is bounded in $W_{0}^{1,(\eta_{j})}(\mathcal{D})$ (or equivalently in $W_{0}^{1,(p_{j})}(\mathcal{D})$), we can apply H\"older's inequality to obtain, for any measurable subset  $E\subset \mathcal{D}$, the estimate
\begin{align*}
\int_{E}b(x)\vert \partial_{j}u_{n}\vert^{p_{j}-2}\partial_{j}u_{n}\;dx&\leq \beta \int_{E}\vert \partial_{j}u_{n}\vert^{p_{j}-1}\;dx\\
&\leq \beta\left(\int_{\mathcal{D}}\vert \partial_{j}u_{n}\vert^{\eta_{j}}\;dx\right)^{\frac{p_{j}-1}{\eta_{j}}}\vert E\vert^{1-\frac{p_{j}-1}{\eta_{j}}}\\
&\leq \beta C\vert E\vert^{\frac{\eta_{j}-p_{j}+1}{\eta_{j}}}.
\end{align*}
This indicates that the sequence  $\{b(x)\vert \partial_{j}u_{n}\vert^{p_{j}-2}\partial_{j}u_{n}\}_{n}$  is equi-integrable. By applying Vitali's theorem, we infer that for every test function $\varphi \in C_{0}^{1}(\mathcal{D})$, the result is
\begin{align}
\label{P3}
\lim\limits_{n\rightarrow+\infty}\sum_{j\in\mathcal{E}}\int_{\mathcal{D}}b(x)\vert \partial_{j}u_{n}\vert^{p_{j}-2}\partial_{j}u_{n}\partial_{j}\varphi\;dx=\sum_{j\in\mathcal{E}}\int_{\mathcal{D}}b(x)\vert \partial_{j}u\vert^{p_{j}-2}\partial_{j}u \partial_{j}\varphi\;dx.
\end{align}
Alternatively, by Lemma \ref{LL2} and \eqref{P1}, we know that the sequence  $\{u_{n}^{q}\vert \partial_{j}u_{n}\vert^{p_{j}-2}\partial_{j}u_{n}\}_{n}$ converges weakly to $u^{q}\vert \partial_{j}u\vert^{p_{j}-2}\partial_{j}u$ in $L^{\rho_{j}}(\mathcal{D})$, for every $1<\rho_{j}<p'_{j}$ and for all indices $j\in\mathcal{E}$. By combining this weak convergence with \eqref{P3}, we derive
\begin{align}
\label{P4}
&\lim\limits_{n\rightarrow+\infty}\sum_{j\in\mathcal{E}}\int_{\mathcal{D}}\left( b(x)+u_{n}^{q}\right)\vert \partial_{j}u_{n}\vert^{p_{j}-2}\partial_{j}u_{n}\partial_{j}\varphi\;dx\nonumber\\
&\quad=\sum_{j\in\mathcal{E}}\int_{\mathcal{D}}\left( b(x)+u^{q}\right)\vert \partial_{j}u\vert^{p_{j}-2}\partial_{j}u \partial_{j}\varphi\;dx, \quad \forall\; \varphi \in C_{0}^{1}(\mathcal{D}).
\end{align}
Regarding the limit of the term on the right-hand side of \eqref{(2.2)}, by Lemma \ref{LL3}, for every $\varphi \in C_{c}^{1}(\mathcal{D})$,
$$\left\vert \frac{f_{n}\varphi}{\left( u_{n}+\frac{1}{n}\right)^{\gamma}}\right\vert\leq \frac{\Vert\varphi\Vert_{L^{\infty}(\mathcal{D})}f}{\tilde{C}^{\gamma}}.$$
Then, from the previous estimate, \eqref{P1}, and by applying Lebesgue's Theorem, we get 
\begin{align}
\label{P5}
\lim\limits_{n\rightarrow+\infty}\int_{\mathcal{D}}\frac{f_{n}}{\left( u_{n}+\frac{1}{n}\right)^{\gamma}}\varphi\;dx=\int_{\mathcal{D}}\frac{f}{u^{\gamma}}\varphi\;dx,\quad \forall \; \varphi\in C_{c}^{1}(\mathcal{D}).
\end{align}
Consider $\varphi \in C_{c}^{1}(\mathcal{D})$ as a test function in \eqref{(2.2)}. By the convergence results \eqref{P4}, \eqref{P5}, and taking the limit as $n \to +\infty$, we obtain 
$$\sum_{j\in\mathcal{E}}\int_{\mathcal{D}}\left( b(x)+u^{q}\right)\vert \partial_{j}u\vert^{p_{j}-2}\partial_{j}u \partial_{j}\varphi\;dx=\int_{\mathcal{D}}\frac{f}{u^{\gamma}}\varphi \;dx,\quad \forall \; \varphi\in C_{c}^{1}(\mathcal{D}).$$
\subsubsection*{Acknowledgements}

\end{document}